\documentclass[11pt]{amsart}
\usepackage{geometry}
\usepackage{hyperref}
\usepackage{MnSymbol}
\usepackage{mathrsfs}
\usepackage{amsmath,amsthm}
\usepackage{tikz}

\usetikzlibrary{
  knots,
  hobby,
  decorations.pathreplacing,
  shapes.geometric,
  calc,
  decorations.markings
}
\usepgfmodule{decorations}
\usepackage{float}
\usepackage{braids}
\usepackage{setspace}
\usepackage{tikz-cd}

\usepackage{wrapfig}
\usepackage{lipsum} 

\theoremstyle{plain}
\theoremstyle{definition}
\newtheorem{Theorem}{Theorem}[section]

\newtheorem{Lemma}[Theorem]{Lemma}
\newtheorem{Proposition}[Theorem]{Proposition}

\newtheorem{Definition}[Theorem]{Definition}
\newtheorem{Remark}[Theorem]{Remark}

\numberwithin{equation}{section}
\DeclareMathAlphabet\mathbb{U}{msb}{m}{n}

\newcommand{\Cay}{\text{\bf Cay}}
\newcommand{\N}{N_{r,k}}

\newcommand{\JJ}{\mathcal{J}_{r,k}}

\newcommand{\RR}{\mathbb{R}}

\newcommand{\Cent}{\text{Cent}}

\newcommand{\ZZ}{\mathbb{Z}}

\newcommand{\Nrk}{N_{r,k}}

\newcommand{\ssspan}{\text{span}}

\begin{document}

\providecommand{\keywords}[1]
{
  \small	
  \textbf{\textit{Keywords---}} #1
}
\title{On the geometry of free nilpotent groups}

\author{Artem Semidetnov}
\author{Ruslan Magdiev}

\address{Laboratory of Continuous Mathematical Education (School 564 of St. Petersburg), nab. Obvodnogo kanala 143, Saint Petersburg, Russia}
\email{artemsemidetnov@gmail.com}
\address{ St. Petersburg State University, 14th Line, 29b, Saint Petersburg, 199178 Russia}
\email{rus.magdy@mail.ru}

\begin{abstract}

In this article, we study geometric properties of nilpotent groups. We find a geometric criterion for the word problem for the finitely generated free nilpotent groups. By geometric criterion we mean a way to determine whether two words represent the same element in a free nilpotent group of rank $r$ and class $k$ by analyzing their behaviour on the Cayley graph of the free nilpotent group of rank $r$ and class $k-1$.
\end{abstract}


\maketitle

\section{Introduction and background}

Nilpotent groups are a long-standing topic in the group theory. In relatively recent years they have started to attract attention from a geometric point of view. As Gromov proved in \cite{pol}, a group as a metric space has polynomial growth if and only if it is virtually nilpotent. It has become a world-recognized example of the geometric uniqueness of nilpotent groups.

As in most varieties of groups, the key to understanding the whole structure is  studying the free objects of the variety. In the case in question, these are the free nilpotent groups. Apparently, any finitely-generated free nilpotent group can be determined by only two positive integers which are the number of generators and the class of nilpotency. The complexity of the geometric structure of these groups increases with these numbers.

The most studied non-abelian free nilpotent group has the smallest determining values among all of the non-abelian free nilpotent groups. These values are 2 and 2. This group is also known as the discrete Heisenberg group $H(\ZZ)$. By its original definition, it coincides with the group of upper triangular matrices over the ring of integers. However, we are interested in the following presentation:

$$H(\ZZ) = \langle a,b \;|\; [a,[a,b]]=[b,[a,b]] = 1\rangle. $$

With respect to this presentation the following fact is true: two words drawn in its Cayley graph end in the same point if and only if special-defined closures of projections on the horizontal plane have the same area. The main result of the present article is to generalize this statement to all finitely generated free nilpotent groups.

  \begin{figure}[h]
  \centering{
\includegraphics[width=110mm]{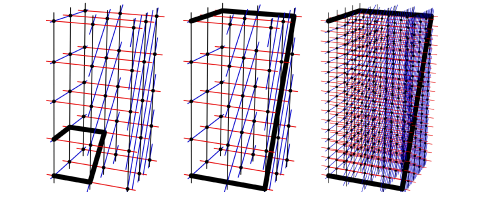}}
\caption{Cayley graph of $H(\ZZ)$ with respect to the generating set $\{a,b,[a,b]\}$.}
\label{CayN2}
\end{figure}

\subsection{Notions}

Let $G$ be a group. For $a,b\in G$ we put $[a,b] = aba^{-1}b^{-1}$ and $[x_1,x_2,\ldots,x_n]$ for $[[[x_1,x_2],x_3]\ldots,x_n]$. For the sake of clarity we define commutator in the degenerate case of $n=1$ as $[x_1] = x_1$. Within a group we will write $[H, K]$ for the subgroup generated by all commutators $[h, k]$ with $h$ ranging over $H \leq G$ and $k$ ranging over $K \leq G$, so that in particular $[G, G]$ is the usual commutator subgroup of $G$. Define a sequence of subgroups inductively as $\gamma_1(G) = G,\; \gamma_s(G) = [\gamma_{s-1}(G),G]$. The group is said to be nilpotent of class $k$ if this sequence converges to trivial subgroup in exactly $k+1$ steps:
$$\{1\} = \gamma_{k+1}(G)\vartriangleleft \gamma_k(G) \vartriangleleft\ldots\vartriangleleft \gamma_3(G)\vartriangleleft \gamma_2(G) \vartriangleleft \gamma_1(G) = G.$$

For $r,k \in \mathbb{N}$ define the free nilpotent group of rank $r$ and class $k$ as the quotient group $F_r/\gamma_{k+1}(F_r)$, where $F_r$ is a unique free non-abelian group of rank $r$. For a basis of $F_r$, say, $\{s_1,\ldots, s_r\}$ put $\JJ =\{ [s_{i_1},\ldots,s_{i_k}] \; | \; \forall i_j\}$. Using this notion we define the standard presentation for a free nilpotent group of rank $r$ and class $k$:
$$N_{r,k} = \langle s_1,s_2,\ldots, s_r \; | \; \mathcal{J}_{r,k+1} \rangle$$

All of the definitions of geometric-group-theory terms may be found in \cite{hours}. However, here we want to put emphasis on the following construction. Let $G$ be a group and $S$ be a fixed set of generators for this group. For a given word $w$ over an alphabet $S \cup S^{-1}$ there's a way to 'draw' this word on $\Cay(G,S)$. It consists of elements represented by prefixes of $w$ which are connected by edges in a consequential order. 

Obviously, two words represent the same element in the group, if and only if their curves end in the same vertex on the Cayley graph. You may see the curves that correspond to the words $aba^{-1}b^{-1}, a^2b^2a^{-2}b^{-2}$ and $a^4b^4a^{-4}b^{-4}$ in the discrete Heisenberg group on the figure \ref{CayN2}. 

In this article, the idea of a word in $F(S)$ being able to represent elements in all of the groups with generating sets equinumerous to $S$ is used several times. That is why for $w\in F_r$ denote $[w]_k$ as the element of $N_{r,k}$ represented by $w$.

\subsection{Known results on $\N$}
Going back to free nilpotent groups, throughout this paper we will consider them only with the standard generating sets, so instead of $\Cay(\N,\{s_1,\ldots,s_r\})$ we will write just $\Cay(\N)$. Now let's pick $r,k>1$. We avoid considering cases $r=1$ or $k=1$ because they are abelian. There is a very helpful normal form for elements of $\N$ that uses so-called Mal'tsev bases. A precise definition is given in \cite{Random} but in our special case, we will just call this name sets $\JJ$. Put $\Delta_{r,k} = |\JJ|$ and denote $c(r,k) := |\Delta_{r,1}| + \ldots +|\Delta_{r,k}|$.

Due to the fact that a Mal'tsev basis $\mathcal{J}_{r,j}$ is equinumerous to the set $\{1,\ldots, \Delta_{r,j}\}$ we can order these sets linearly. Choose such an order $\preccurlyeq$. Define the canonical form of an element $g$ as 
\begin{equation}\label{word pres}
    u_g = \prod_{w_i \in \mathcal{J}_{r,k}}w_i^{\alpha^{k}_i} \cdot\ldots\cdot \prod_{w_i \in \mathcal{J}_{r,2}}w_i^{\alpha^2_i}\cdot \prod_{s_i \in S}s_i^{\alpha^1_i}
\end{equation}
where the products are ordered in correspondence to $\preccurlyeq$.
Though it may seem confusing, it is just a product of elements of $\cup \mathcal{J}_{r,i}$ for $i = \overline{1,k}$ in some powers $\alpha_i^j$, uniquely defined by the element $g$.

\subsection{Geometric characteristic for words and the main result.}

Here we define a geometric characteristic for words in $F_r$ presenting elements of $N_{r,k}$. The main result, that is {\bf Theorem A} draws a connection between this characteristic and the word problem for $N_{r,k}$.

The (\ref{word pres}) product gives rise to a coordinate form of elements of $\N$: each of them can be associated with an array of integers, given by powers of the elements of $\cup \mathcal{J}_{r,i}$\footnote{in the example above these are $\alpha_i^j$}, with $i=\overline{1,k}$. As $c(r,k) = \Delta_{r,k}+c(r,k-1)$ for an element $g$ we will simply write
$$g = (T_1,T_2,\ldots, T_{\Delta_{r,k}},t_{c(r,k-1)},\ldots, t_1)$$
and call $T_i$ and $t_j$ for $i=\overline{1,\Delta_{r,k}}$ and $j =\overline{1,c(r,k-1)}$ coordinates. We distinguish them by capitalization because of the folowing fact. Let $\pi_k$ be the canonical epimorphism $N_{r,k}\to N_{r,k-1}$ which kernel is $\gamma_k(N_{r,k})$. It is easy to see that under that mapping 
$$(T_1,T_2,\ldots, T_{\Delta_{r,k}},t_{c(r,k-1)},\ldots, t_1) \mapsto (t_{c(r,k-1)},\ldots, t_1)$$

The coordinate system defines a way to immerse $\Cay(\N)$ into $\RR^{c(r,k)}$ via the map identifying coordinates in $N_{r,k}$ and $c(r,k)$-tuples in $\RR^{c(r,k)}$.

As there is a chosen order $\preccurlyeq$ on $\JJ$ denote $K_i$ the $i$-th element in $\JJ$. It is easy to see that $K_i$ represents the element
$$(0,\ldots, 0,1,0,\ldots,0)$$
where the only non-zero coordinate is the $i$-th.

\begin{Definition}
Let $w$ be a word in $F_r$ and $k\in \mathbb{N}$. Let $i\in \{1,\ldots,\Delta_{r,k}\}$.  Then $K_i = [k,s]$ where $s\in \{s_1,\ldots, s_r\}$ and $k\in \mathcal{J}_{r,k-1}$. Denote by $\overline{s}, \overline{k}\in \RR^{c(r,k-1)}$ the images of $s,k$ under the immersion defined above.
Let ${\bf W}\subset \RR^{c(r,k)}$ be the curve for the word $w$ on the $\Cay(\N)$. We say that the orthogonal projection of ${\bf W}$ on the plane $\ssspan(\overline{s}, \overline{k})$ is the projection of the word $w$ on the $i$-th coordinate.
\end{Definition}

Sometimes we will abuse notation and say that the projection is being considered not on the $i$-th coordinate but on the $K_i$-th coordinate.

Now we define {\it area} for a projection of a word $w$ on the $i$-th coordinate for some $i$. Obviously, area enclosed be some curve can be defined only in the case when the said curve is closed. It is explained below that the curve for a word $w$ on $\Cay(\N)$ projected on $\Cay(N_{r,k-1})$ is closed if and only if it represents an element of $\gamma_k(\N)$. 

Define {\bf closure} of a word $w$ representing $g \in \N$ as the $u_{\pi_k(g^{-1})}$. As it is explained below the curve's projection for $\omega\cdot u_{\pi_k(g^{-1})}$ projection on $\Cay(N_{r,k-1})$ is closed. Call the oriented area of the polygon defined as the projection of $\omega\cdot u_{\pi_k(g^{-1})}$ onto the $i$-th coordinate the area of $w$'s projection onto the $i$-th coordinate.

Note that closure of $K_i$ is the empty word. For each $K_i \in \JJ$ denote $C_i$ for its projection's area on the $i$-th coordinate.

Our main result is the following fact: 

\begin{spacing}{1.5}
\end{spacing}

\noindent {\bf Theorem A}.
A word $w \in F_r$ represents an element $g\in \Nrk$ with coordinates $(T_1,T_2,\ldots,T_{\Delta_{r,k}},t_{c(r,k-1)},\ldots,t_1)$ if and only if $w$ represents an element $(t_{c(r,k-1)},\ldots,t_1)$ in $N_{r,k-1}$ and areas of its standard closure's projections onto $T_i$-coordinates equal to $C_iT_i$ for $i = \overline{1,\Delta_{r,k}}.$
\begin{spacing}{1.5}
\end{spacing}

 \section{Some properties of $\N$}
 In this section, we elaborate on some properties of free nilpotent groups. Again, pick constants $r,k>1$.
 \begin{Proposition}
 Center of the $\N$ coincides with $\gamma_k(\N)$.
 \end{Proposition}
 
\begin{Remark}\label{closed}
A word $w$ represents a closed curve on $\Cay(N_{r,k-1})$ if and only if $[w]_k \in \gamma_k(\N)$.
\end{Remark}

\begin{Lemma}\label{Area lemma}The following facts are true
\begin{enumerate}
    \item Area of a word $u\in \mathcal{J}_{r,k+1}$ or its inverse projected on any $i$-th coordinate equals zero.
    \item Let $w,v \in F(S)$. If $[w]_k = [v]_k$ then on any $i$-th coordinate areas of their projections coincide.
    \item Let $w\in \JJ$ and $v \in F(S)$ and let their $w$-projection areas be equal to $\Sigma_w$ and $\Sigma_v$ respectively. Then the area of $w$ projection of $vw$ equals to $\Sigma_v+\Sigma_w$.
\end{enumerate}
\end{Lemma}

\begin{Lemma}\label{govno}
If $g\in \Nrk$ has the following coordinates

$$g = (T_1,\ldots, T_{\Delta_{r,k}},0\ldots,0), $$

and each of all its areas on $i$-th projections are zeros, then $g = 1$.
\end{Lemma}

\subsection{Values of $c(r,k)$ and $\Delta_{r,k}$.}
\begin{Proposition}[see \cite{loh} 7, 11.2.2]
Numbers $\Delta_{r,k}$ are given by necklace polynomials: 
$$\Delta_{r,k} = \frac{1}{k}\sum_{d \,| \,k}\mu(d)r^{k/d}$$
Where $\mu$ is the m$\ddot o$bius function.
\end{Proposition}

Also we prove the following recursion property for $c(r,k)$:

\begin{Lemma}\label{crk}
Value $c(r,k)$ satisfies following relation
$$c(r, k) = (r + 1) \cdot c(r, k - 1) -  r \cdot c(r, k - 2)$$
with $c(r,1) = r$, $c(r,2) = r(r+1)/2$ for $r>2$ and $c(2,2)=1$.
As a corollary, $c(2,k) = 2^{k-1}+1$. 
\end{Lemma}

Also, we should include the notion of null sequences. For any group $G = \langle S \; | \; R \rangle$ consider $w\in F(S)$. Word $w$ can be broken down into a product of two words (possibly empty), $w=uv$. Also, for any $r \in R$ and $s\in S$ words $uss^{-1}v$ and $usrs^{-1}v$ represent the same element in $G$ as $w$. Actually, any two words representing the same element in $G$ can be transformed into each other using these alterations. They are called null-sequences (see \cite{hours}).

\subsection{Examples on $N_{2,2}$ and $N_{2,3}$}
The Discrete Heisenberg group is the only group for which the geometric criterion has already been known. It was introduced in \cite{Shapiro} and widely used in \cite{Ruslan, Shapiro, InfiniteG, Rational growth}. Its $\pi_2$ epimorphism is actually the orthogonal projection of $\Cay(N_{2,2},\{a,b\})$ onto $\Cay(\ZZ\times\ZZ,\{a,b\})$ given by the abelenization map. Moreover, the corresponding short exact sequence is split

$$1\longrightarrow \ZZ \longrightarrow N_{2,2}\longrightarrow \ZZ\times\ZZ\longrightarrow1$$

   \begin{figure}[h]
  \centering{
\includegraphics[width=110mm]{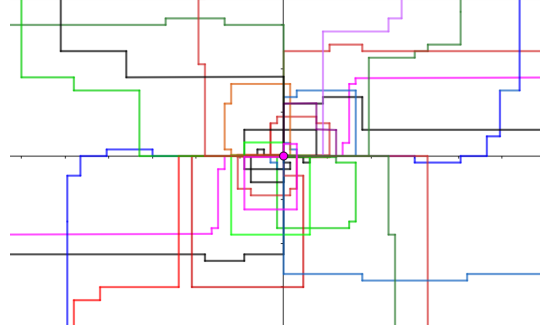}}
\caption{Projections of words in $N_{2,2}$ on  $\Cay(N_{2,1})$.}
\label{Geodesics}
\end{figure}

The criterion in $N_{2,2}$ is quite easy to comprehend: two words $w,v\in F(a,b)$ represent the same element in the discrete Heisenberg group if and only if they, being pictured on $\Cay(\ZZ\times\ZZ)$, end in the same point $(m,n)$ on this graph and the oriented areas of their closures $wa^{-n}b^{-m}, va^{-n}b^{-m}$ coincide.

Now let's consider $N_{2,3}$. The theorem states that we should look at $\Cay(N_{2,2})$. The set $\mathcal{J}_{2,3}$ consists of two elements: $\mathcal{J}_{2,3} = \{[a,a,b], [b,a,b]\}$. Exact computations show that areas of their projections are 2 for $[a,a,b]$ and 1 for $[b,a,b]$. It is easy to see that the closure for a word $w$ ending in $(k,m,n) \in \ZZ^3$ on $\Cay(N_{2,2})$ is $wa^{-n}b^{-m}[a,b]^{-k}$.

 \begin{figure}[h]
 \centering{
\includegraphics[width=110mm]{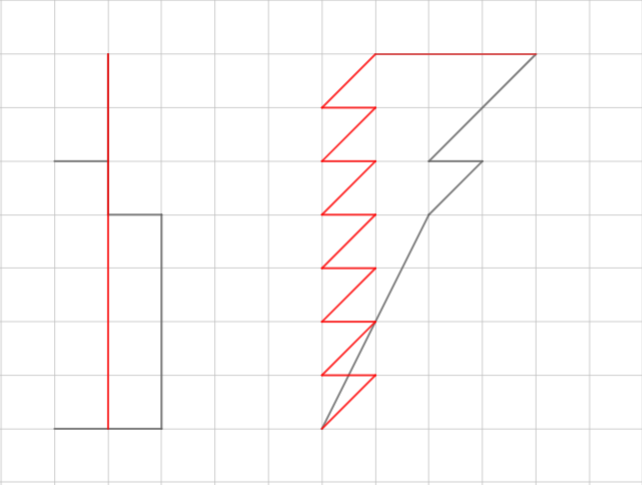}}
\caption{Word $w = aabba^{-1}ba^{-1}b^{-1}abb$ presented on $\Cay(N_{2,2})$ and projected to planes $x=0$ and $y=0$ in $(x,y,z)$ coordinate system.}
\label{Wordss}
\end{figure}
For example, picture \ref{Wordss} shows the word $w =aabba^{-1}ba^{-1}b^{-1}abb$. It is drawn in grey, and red lines represent its closure. It is easy to see that areas of figures are 4 and 10 for projections on first and second coordinates respectively, which is drawn in red has projections with areas 4 and 10. Consequently, $[w]_3 = (8,10,5,4,1)$.

\section{Proofs}

\begin{Remark}\label{area of ug}

For any element 
$$g = (T_1,T_2,\ldots, T_{\Delta_{r,k}},t_{c(r,k-1)},\ldots, t_1)$$

An area of the word $u_g$ projected on $i$-th coordinate equals to $C_iT_i$ for all $i = \overline{1,\Delta_{r,k}}$
\end{Remark}

\begin{proof}[Proof of Lemma \ref{crk}]
This proof consists of two computations:

    We define {\it graph of commutators} $\Gamma_{r,k}$. Its set of vertices is $V(\Gamma_{r,k}) = \cup \mathcal{J}_{r,j}$ for $j = \overline{1,k}$ and the commutator $[s_1,\ldots,s_j]$ for $j\leq k$ is connected with 
    \begin{enumerate}
        \item $s_1$
        \item $[s_2,\ldots,s_j]$
        \item $[s,s_1,s_2,\ldots,s_j]$ for all $s\in S$ if $j < k$.
    \end{enumerate}
    You can see $\Gamma_{2,4}$ on picture \ref{commies}. Obviously, $c(r,k) = |\Gamma_{r,k}|$. Let's compute number of elements that adds to $\Gamma_{r,k}$ in $\Gamma_{r,k+1}$. (Obviously, there's an embedding $i:\Gamma_{r,k}\rightarrow\Gamma_{r,k+1}$). It is easy to see that any element of $\Gamma_{r,k+1}\backslash i(\Gamma_{r,k})$ is connected with exactly two elements: one in $i(\Gamma_{r,k})\backslash i(\Gamma_{r,k-1})$ and one in $S$. So, number of elements in $\Gamma_{r,k+1}\backslash i(\Gamma_{r,k})$ equals to number of edges in complete bipartite graph $K_{m,n}$, where $m = |i(\Gamma_{r,k})\backslash i(\Gamma_{r,k-1})|$ and $n = |S| = r$. Thus,
    
 \begin{multline*}
       c(r,k) = r\cdot|\Gamma_{r,k}\backslash \Gamma_{r,k-1}| + c(r,k-1) =\\= r(c(r,k-1) - c(r,k-2)) + c(r,k-1) =(r+1)\cdot c(r,k-1) - r\cdot c(r,k-2)
 \end{multline*}  
 
 For $k=2$ put $c_k:=c(2,k)$. Then it admits linear recurrent relation of a form $c_k = 3c_{k-1}-2c_{k-2}$.It has corresponding characteristic polynomial $q^2-3q+2$. It has two roots: $2,1$. Then any sequence $d_n = A\cdot 2^k+B$ has the same kind of linear recurrent relation. So, by solving system of equations $d_1 = c_1,\, d_2 = c_2$ for $A,B$ we will find that form for $c_k$. Recall that $c_1 = 2, c_2 = 1$.
 $$\begin{cases}
 2 = 2A+B\\
 1=4A+B
 \end{cases}$$
 It is easy to see that $A = 1/2, B = 1$. So we have $c_k = c(2,k) = 2^{k-1}+1$.

   \begin{figure}[h]
  \centering{
\includegraphics[width=110mm]{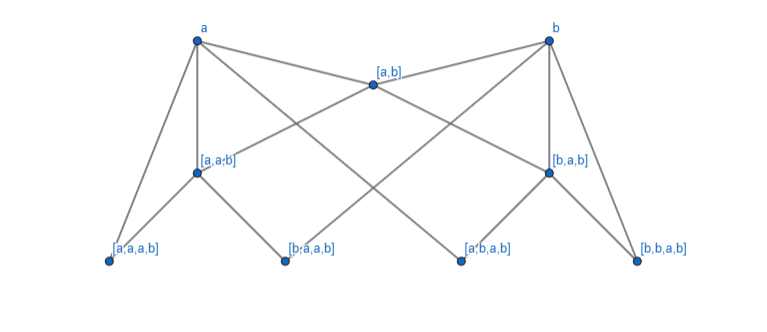}}
\caption{Graph of commutators $\Gamma_{2,4}$.}
\label{commies}
\end{figure}
\end{proof}


\begin{proof}[Proof of Lemma \ref{closed}]
Obviously, the word $w$ represents a closed curve on $\Cay(N_{r,k-1})$ if and only if $[w]_{k-1}$ is the identity element in $N_{r,k-1}$. It is equivalent to the fact that $[w]_k \in \ker(\pi_k)$. As it was noticed above, $\ker(\pi_k)=\Cent(\N)$.
\end{proof}

\begin{proof}[Proof of Lemma \ref{Area lemma}]\ \\
\begin{enumerate}
    \item Note that any $u\in \mathcal{J}_{r,k+1}$ can be seen as $[s,v]$ for some $s\in S$ and $v\in \mathcal{J}_{r,k}$. So, $u = svs^{-1}v^{-1}$. By multiplication in $\N$ it is easy to see that curve of $svs^{-1}$ ends in {\bf 0} and segments corresponding to $s$ and $s^{-1}$ on projection coincide. Moreover, projections of $v$ and $v^{-1}$ are isometric in $\RR^2$ and oriented in different ways. Thus, area of the projection is $0$.
    
    \item Let $[w]_k = [v]_k$. Then $[wv^{-1}]_k = 1$, which means (see \cite{hours}) that there's a null-sequence of $wv^{-1}$ leading to the empty word. Now we prove that any null-transformation does not change area of a word's projection. Let $u = u_1u_2$ and $s\in S$. Obviously, $u_1ss^{-1}u_2$ as a curve represents curve $u_1u_2$ with glued to $[u_1]_k$ segment that ends in $[u_1s]_k$, so its area did not change. The second transformation is $u \mapsto u_1srs^{-1}u_2$, where $r\in \mathcal{J}_{r,k+1}$ is a relation word. It is easy to see that its curve consists of $u_1u_2$ curve with curve of $srs^{-1}$ glued to $[u_1]_k$. Note, that it has zero area of projection because of the previous point in this lemma. Note, that inverse null-transformations need not to be analyzed separately, because we have shown that images' area always coincide with preimages'.
    
    \item As $[w]$ lies in $ \Cent(\N)$, words $vw$ and $wv$ represent the same element in $\N$. Obviously, area of $wv$ equals to $\Sigma_w+\Sigma_v$. So using previous point of this lemma we obtain what we desired so hardly. 
    
\end{enumerate}

\end{proof}

\begin{proof}[Proof of the lemma \ref{govno}]
Suppose $w$ represents a non-identity element $g\in \Nrk$ and it has above-mentioned coordinates. Then consider $u_g$. It has areas that are equal to $C_iT_i$ on the corresponding planes. Then by lemma \ref{Area lemma} we get that areas of $w$ coincide with these numbers. Hence, they are no zeros. 
\end{proof}

\begin{proof}[Proof of the main theorem]\ \\
Start with the necessity. As $$[w]_k=g=(T_1,T_2,\ldots, T_{\Delta_{r,k}},t_{c(r,k-1)},\ldots, t_1)$$
its image under $\pi_k$ represents $(t_{c(r,k-1)},\ldots, t_1)\in N_{r,k-1}$. That's the first part. Next, consider $u_g$. As $[u_g]_k=[w]_k$, there exists a null-sequence that transforms $u_g$ to $w$. As it was noted above, areas of projections on $T_i$ coordinates of $u_g$ equal to $C_iT_i$ with $i=\overline{1,\Delta_{r,k}}$. Also, null-sequences do not change these areas. Consequently, an area of projection of $w$ on $T_i$ coordinate equals to $C_iT_i$ for every $i=\overline{1,\Delta_{r,k}}$\\

Sufficiency. Let 
$$[w]_k=(T_1',\ldots,T_{\Delta_{r,k}}',t_{c(r,k-1)},\ldots,t_1).$$ Let us prove that $w$ represents the same element as $u_g$. Consider the word $v = wu_g^{-1}$. The following holds for it:
\begin{multline*}
    [v]_k = (T_1',\ldots,T_{\Delta_{r,k}}',t_{c(r,k-1)},\ldots,t_1)(T_1,T_2,\ldots, T_{\Delta_{r,k}},t_{c(r,k-1)},\ldots, t_1)^{-1} = \\ = (T_1',\ldots,T_{\Delta_{r,k}}',0,\ldots,0)\cdot(0,\ldots,0,t_{c(r,k-1)},\ldots,t_1)(T_1,\ldots,T_{\Delta_{r,k}},0,\ldots,0)^{-1}\cdot(0,\ldots,0,t_{c(r,k-1)},\ldots,t_1)^{-1} = \\ = 
    (T_1',\ldots,T_{\Delta_{r,k}}',0,\ldots,0)(T_1,\ldots,T_{\Delta_{r,k}},0,\ldots,0)^{-1}
    \end{multline*}
    Consequently, it suffices to check that the last product equals to the identity.
    
    By the lemma \ref{Area lemma} and the condition on areas of $w$ stated in the theorem we can state that all the areas of projections of the closed curve $wu_g^{-1}$ are equal to zero. Then, by the \ref{govno} we get $[wu_g^{-1}]_k = 1$. Then $[w]_k = g$.

\end{proof}


\begin{thebibliography}{99}
\bibitem{Ruslan} Ilya Alexeev \& Ruslan Magdiev, {\it The language of geodesics for the discrete Heisenberg group}, arXiv:1905.03226

\bibitem{Shapiro} M. Shapiro, {\it A geometric approach to the almost convexity and growth of some nilpotent groups}, Math. Ann., 285, 601624 (1989).
 
 \bibitem{InfiniteG} A. M. Vershik, A. V. Malyutin, {\it Infinite geodesics in the discrete Heisenberg group}, J. Math. Sci. (N. Y.), 232:2 (2017), 121128.
 
 \bibitem{Length}  Sebastien Blacher, {\it Word Distance On The Discrete Heisenberg Group}, Colloquim mathematicum, 2003, vol 95, 1.
 
\bibitem{hours} Clay, M., \& Margalit, D. (Eds.). (2017). {\it Office Hours with a Geometric Group Theorist}. Princeton University Press

\bibitem{geod growth}Brönnimann, J.M. (2016). {\it Geodesic growth of groups}.

\bibitem{pol} M. Gromov,{\it Groups of Polynomial growth and Expanding Maps}, Publications mathematiques I.H.É.S., 53, 1981


\bibitem{highway} Meijke Balay {\it Introduction to the Heisenberg Group}, ISBN-10: 3764391855

\bibitem{Rational growth} M. Duchin, M. Shapiro,{\it The Heisenberg group has rational growth in all generating sets}, 2017

\bibitem{Random} M.Cordes, M.Duchin, Y.Duong, Meng-Che Ho, A.P. Sanchez, {\it Random nilpotent groups I}, arXiv:1506.01426


\bibitem{Maltsev} A. I. Mal'tsev, {\it On a class of homogeneous spaces}, Izv. Akad. Nauk SSSR Ser. Mat., 13:1 (1949), 9–32
    
\bibitem{loh} Marshall Hall, Jr., The theory of groups. Second edition. American Mathematical Society, 1976

\end{thebibliography}
\end{document}